\documentclass[10pt,a4paper]{amsart}

\theoremstyle{plain}
\newtheorem{theorem}{Theorem}[section]
\newtheorem{lemma}[theorem]{Lemma}
\newtheorem{proposition}[theorem]{Proposition}
\newtheorem{corollary}[theorem]{Corollary}

\theoremstyle{definition}

\theoremstyle{remark}
\newtheorem{remark}[theorem]{Remark}

\newcommand{\Ps}{\mathbb{P}}

\def\bin #1#2 {\left( \matrix { #1 \cr #2 \cr } \right) }

\begin{document}

\title[Pl\"ucker--Clebsch formula in higher dimension]
{Pl\"ucker--Clebsch formula in higher dimension}


\author{Ciro Ciliberto }
\address{Universit\`a di Roma \lq\lq Tor Vergata\rq\rq, Dipartimento di Matematica,
Via della Ricerca Scientifica, 00133 Roma, Italy.}
\email{cilibert@axp.mat.uniroma2.it}

\author{Vincenzo Di Gennaro }
\email{digennar@axp.mat.uniroma2.it}

\abstract Let $S\subset\Ps^r$ ($r\geq 5$) be a nondegenerate,
irreducible, smooth, complex, projective surface of degree $d$. Let
$\delta_S$ be the number of double points of a general projection
of $S$ to $\Ps^4$. In the present paper we prove that $
\delta_S\leq{\binom {d-2} {2}}$, with equality if and only if
$S$ is a rational scroll. Extensions to higher dimensions are discussed.



\endabstract

\maketitle
\section{Introduction.}
Let $X$ be a nondegenerate, irreducible, smooth,
projective variety of dimension $n\geq 1$ 
in the complex projective space $\Ps^r$ with $r\geq 2n+1$.
A general projection to $\Ps^{s}$, with $2n+1\leq s\leq r$, induces an
isomorphism of $X$ with its image. A general 
projection to $\Ps^{2n}$ induces an isomorphism of
$X$ with its image, except for a finite set of points of $X$,
which correspond to a certain number $\delta_X$ of \emph{improper double points}
of the image, i.e. double points with tangent cone formed by two
linear spaces of dimension $n$ spanning $\Ps^{2n}$. 
The \emph{double point formula} (see \cite{Fulton}, pg. 166)
expresses $\delta_X$ in terms of  invariants of $X$. When $X$ is a
curve of genus $g$ and degree $d$, the double point formula
says that $\delta_X ={\binom {d-1} {2}}-g$. This is the
classical \emph{Pl\"ucker--Clebsch formula}. In particular
$\delta_X\leq{\binom {d-1} {2}}$ and equality holds if and only
if $X$ is a rational curve.

In \S \ref {proof} of this paper we prove a similar result for
surfaces:

\begin{theorem}\label{clebsch}
Let $S\subset\Ps^r$, with $r\geq 5$,  be a nondegenerate, irreducible,
smooth, projective surface of degree $d$. Then
\begin{equation}\label{eq:clebch}
\delta_S\leq{\binom {d-2} {2}}
\end{equation}
with equality if and only if $S$ is a rational scroll. 
\end{theorem}

In \S ~\ref {highdim} we examine the higher dimensional case,
in the attempt of proving a similar theorem, but we obtain only
partial results (see Proposition \ref{formula} and Remark
\ref{finalr}, $(iii)$ and $(iv)$). In this context there is some
evidence supporting a  conjecture to the effect that
$\delta_S\leq{\binom {d-2} {2}}-g$, where $g$ is the
sectional genus of the surface $S$. We are able to prove this only
in some cases, e.g. when the \emph{Kodaira dimension} $\kappa(S)$ of $S$ is not positive
(see Remark \ref{finalr}, $(i)$ and $(ii)$). However in \S \ref {sec:strong} we show the
intermediate inequality $\delta_S\leq{\binom {d-2}
{2}}-\frac{g}{2}$ (see Theorem \ref{clebsch2}).

\section{Proof of Theorem \ref{clebsch}}\label {proof}

\begin{remark}\label{dpformula}
$(i)$  We say
that $S\subset \Ps^r$ is a {\it{scroll}} if it is a
$\Ps^1$--bundle over a smooth curve $C$, and the restriction
of $\mathcal O_S(1)$ to a fibre is $\mathcal O_{\Ps^1}(1)$.  If
$C\cong \Ps^1$, the scroll is said to be \emph{rational}.  In this case
$S$ is isomorphic,  via projection, to a
\emph {rational normal scroll} $S'\subset{\mathbb P^{r'}}$ with $r'={\rm deg}(S)+1$ (see \cite {EHarris2}).

$(ii)$ Let $H_S$ be the general hyperplane section of a smooth surface $S\subset \mathbb P^r$.
The line
bundle $\mathcal O_S(K_S+H_S)$ is spanned, unless $S$ is either a scroll, or $S\cong \mathbb P^2$ and $\mathcal O_S(H_S)\cong \mathcal O_{\mathbb P^2}(i)$, with $1\leq i\leq 2$,  in which
cases  $\mathcal O_S(K_S+H_S)$ is not effective (see \cite {SVDV}, Theorem (0.1)). 

$(iii)$ Assume $S\subset \mathbb P^r$, with $r\geq 5$, is a smooth, irreducible, nondegenerate surface. The double point formula says that
$$
\delta_S=\frac {d(d-5)} 2-5(g-1)+6\chi(\mathcal O_S)-K^2_S,
$$
where $g$ and $K_S$ denote the \emph{sectional genus} and a \emph{canonical
divisor} of $S$, and $\chi(\mathcal O_S)=1+p_a(S)$, where $p_a(S)$
is the \emph{arithmetic genus} of $S$. Hence
\begin{equation}\label{doublepoint}
\delta_S-{\binom {d-2} {2}}=6\chi(\mathcal O_S)-K^2_S-5(g-1)-3.
\end{equation}
Note that $\delta_S=0$ if and only if $S$ is \emph{secant defective}, i.e.~$\dim({\rm Sec}(S))<5$, where ${\rm Sec}(X)$ denotes the \emph{secant variety} of a variety $X$, i.e. the Zariski closure of the union of all lines spanned by distinct points of $X$.
A  theorem of Severi's implies that a smooth surface is secant defective if and only if it is the \emph{Veronese surface} of degree 4 in $\mathbb P^5$ (see \cite {Ciro, Zak}).

\end{remark}

\begin{proposition}\label{nspanned} In the above setting, 
if $\mathcal O_S(K_S+H_S)$ is not spanned, then \eqref {eq:clebch} holds, with equality if and only if $S$ is
a rational scroll.
\end{proposition}
\begin{proof}  If $S$ is the Veronese surface of degree 4 in $\mathbb P^5$,
we have $\delta_S=0$ and the assertion holds.
Otherwise $S$ is a scroll, hence $\chi(\mathcal O_S)=1-g$, $K^2_S=8(1-g)$.
Plugging into (\ref{doublepoint}) we obtain $\delta_S-{\binom
{d-2} {2}}=-3g$ and the assertion follows.
\end{proof}

In order to prove
Theorem \ref{clebsch}, it suffices to prove the following:

\begin{proposition}\label{spanned}
Assume that $\mathcal O_S(K_S+H_S)$ is spanned. Then \eqref {eq:clebch} holds with strict inequality.
\end{proposition}

We collect some preliminaries in the following lemma.

\begin{lemma}\label{g} Let $S\subset\Ps^r$ ($r\geq 5$) be a nondegenerate, irreducible,
smooth, projective surface of degree $d$ and sectional genus $g$. Denote by $e$ the
\emph{index of speciality} of the general hyperplane section $H_S$ of $S$, i.e. $e:=\max\,\{t\in\mathbb Z\,:\,H^1(H_S,\mathcal
O_{H_S}(t))\neq 0\}$. \medskip

$(i)$ If $g > d-2$ then \eqref {eq:clebch} holds with strict inequality.

\medskip

$(ii)$ If $p_a(S)\leq 0$ then \eqref {eq:clebch} holds,
with equality if and only if $S$ is a rational scroll.

\medskip

$(iii)$ If $e\neq 1$ then \eqref {eq:clebch} holds, with
equality if and only if $S$ is a rational scroll.

\medskip

$(iv)$ If $e=1$ then $p_a(S)\leq g-\frac{d}{2}$.

\end{lemma}
\begin{proof}
$(i)$ Let $\mathcal N_{S,\Ps^r}$ be the normal bundle of $S$ in
$\Ps^r$. We have 
\begin{equation}\label{eq:vS}
c_2(\mathcal
N_{S,\Ps^r}(-1))=2(K_S^2-6\chi(\mathcal O_S))+4K_SH_S+6H_S^2\geq 0
\end{equation}
the last inequality holding because $\mathcal N_{S,\Ps^r}(-1)$ is spanned. Then
$6\chi(\mathcal O_S)-K_S^2\leq 4(g-1)+d$. Plugging into
(\ref{doublepoint}) we obtain
\begin{equation}\label{positivity}
\delta_S-{\binom {d-2} {2}}\leq d-g-2
\end{equation}
and the assertion follows.

\smallskip

$(ii)$ By Proposition \ref{nspanned} it suffices to prove that:
{\it{if $p_a(S)\leq 0$ and $\mathcal O_S(K_S+H_S)$ is spanned then
\eqref {eq:clebch} holds with strict inequality}}. If $\mathcal O_S(K_S+H_S)$
is spanned then $(K_S+H_S)^2\geq 0$, therefore $K_S^2\geq
d-4(g-1)$. Plugging into (\ref{doublepoint}) we get
\begin{equation}\label{rifconj}
\delta_S-{\binom {d-2} {2}}\leq 6p_a(S)-g-d+4.
\end{equation}
If $p_a(S)\leq 0$ then  $\delta_S-{\binom {d-2}
{2}}\leq -g-d+4$, which is negative unless $d=4$, in which case $(ii)$ is trivial.

\smallskip

$(iii)$ From the Poincar\'e residue sequence tensored with
$\mathcal O_S(i-1)$
$$
0\to \omega_S(i-1)\to \omega_S(i)\to \omega_{H_S}(i-1)\to 0
$$
we get
\begin{equation}\label{pg}
p_g(S)\leq \sum_{i=1}^{+\infty}h^1(H_S,\mathcal O_{H_S}(i)),
\end{equation}
where $p_g(S):=h^0(S,\mathcal O_S(K_S))$ is the \emph{geometric genus} of $S$. If $e\leq 0$
from (\ref{pg}) we deduce $p_a(S)\leq p_g(S)\leq 0$, and 
 \eqref {eq:clebch} holds by $(ii)$. If $e\geq 2$
then $2g-2-2d\geq 0$, i.e. $g\geq d+1$, and by $(i)$ we have  \eqref {eq:clebch} with strict inequality. When $\delta_S={\binom {d-2} {2}}$ the
previous argument yields $e\leq 0$, so $p_a(S)\leq 0$, and
$S$ is a rational scroll by $(ii)$.

\smallskip

$(iv)$ If $e=1$,  from (\ref {pg}) we deduce $p_a(S)\leq
h^1(H_S,\mathcal O_{H_S}(1))$. By Clifford Theorem we have
$h^0(H_S,\mathcal O_{H_S}(1))\leq 1+\frac{d}{2}$. Therefore
$p_a(S)\leq h^1(H_S,\mathcal O_{H_S}(1))=h^0(H_S,\mathcal
O_{H_S}(1))-(1-g+d)\leq g-\frac{d}{2}$. \end{proof}

We are now in position to prove Proposition \ref{spanned}.

\begin{proof}[Proof of Proposition \ref{spanned}]
Consider the \emph{adjunction map} $\phi: S\to \Ps^R$ defined by
$|K_S+H_S|$. By \cite{SVDV}, Theorem (0.2) and \cite{BS}, Lemma
1.1.3, Lemma 10.1.1 and Theorem 10.1.3,  we know that if
$\dim(\phi(S))\leq 1$ then $S$ is birationally ruled. Hence
$p_a(S)\leq 0$, and Proposition \ref{spanned} follows by Lemma \ref{g}, $(ii)$.

Suppose $\dim(\phi(S))=2$. By Lemma \ref{g}, $(ii)$, we may assume
$p_a(S)>0$. Hence $\kappa(S)\geq 0$, so we may apply \cite{BS}, Lemma
10.1.2, which ensures that $(K_S+H_S)^2\geq 2(p_a(S)+g-2)$, i.e.
$ K_S^2\geq 2\chi(\mathcal O_S)-2(g-1)+d-4$. Hence by
(\ref{doublepoint}) it suffices to prove that $ 2\chi(\mathcal
O_S)-2(g-1)+d-4>6\chi(\mathcal O_S)-5(g-1)-3$, i.e. that
\begin{equation}\label{ultimariduzione}
4p_a(S)<3g+d-8.
\end{equation}
To prove this, note that by  Lemma \ref{g}, $(i)$, $(iii)$ and $(iv)$, we may
assume $g\leq d-2$,  $e=1$ and $p_a(S)\leq g-\frac{d}{2}$. So we
have
$$
4p_a(S)\leq 4g-2d<3g+d-8.
$$
This proves (\ref{ultimariduzione}),  concluding the proof of
Proposition \ref{spanned} and of Theorem \ref{clebsch}.
\end{proof}

\begin{remark} Theorem \ref {clebsch} holds even if $S\subset \mathbb P^4$ is an
irreducible, nondegenerate surface of degree $d$ with only  $\delta_S$ improper double points as singularities. 
Again \eqref {eq:clebch} holds, with equality if and only if $S$ is the projection
in $\mathbb P^4$ of a smooth rational normal scroll in $\mathbb P^{d+1}$.
There is no difficulty in adapting the above argument, hence we will not dwell on this. 
\end{remark}

\section{Results in higher dimension}\label{highdim}

Let $X\subset\Ps^r$ be a nondegenerate, irreducible, smooth,
projective variety of dimension $n$ and degree $d$, with
$r\geq 2n+1$. In view of Pl\"ucker--Clebsch formula in the 1--dimensional case and of Theorem
\ref{clebsch} for $n=2$, one may ask whether in general
\begin{equation}\label{extension}
\delta_X\leq{\binom {d-n} {2}},
\end{equation}
with equality if and only if $X$ is a \emph{rational scroll}, i.e.
the projection of a \emph{rational normal scroll} in
$\Ps^{d+n-1}$  (see \cite{EHarris2}).

\begin{remark}\label{rem:scrolls}
When $X$ is a {\it{scroll}},
i.e. when $X$ is a $\Ps^{n-1}$--bundle over a smooth curve of genus
$g$ and the restriction of $\mathcal O_X(1)$ to a fibre is
$\mathcal O_{\Ps^{n-1}}(1)$, then the double point formula gives $\delta_X-{\binom {d-n}
{2}}=-{\binom {n+1} {2}}g$.
\end{remark}

Before proceedings further, 
let us recall a geometric 
interpretation of  
$$v_X:=c_n(\mathcal N_{X,\Ps^r}(-1))$$
where $\mathcal N_{X,\Ps^r}$ is the normal bundle of
$X$ in $\Ps^r$. Note that $v_X\geq 0$ because $\mathcal N_{X,\Ps^r}(-1)$ is spanned.

Let ${\rm Tan}(X)$ be the \emph{tangential variety} of $X$, i.e. the Zariski closure of the union of all tangent spaces to $X$ at 
smooth points, which makes sense even if $X$ is singular. Denote by $t_X$ the degree of ${\rm Tan}(X)$. One has $\dim({\rm Tan}(X))\leq 2n$ and, if strict inequality holds, $X$ is called \emph{tangentially defective}, whereas $X$ is called \emph{secant defective} if $\dim({\rm Sec}(X))<2n+1$.

Note that
${\rm Tan}(X)$ is contained in ${\rm Sec}(X)$. If $X$ is smooth, then it is secant defective if and only if ${\rm Tan}(X)={\rm Sec}(X)$ (see \cite {Zak}). Hence, if $X$ is smooth and tangentially defective, then it is also secant defective, but the converse does not hold in general.

If $X$ is not tangentially defective, there are finitely many tangent spaces to $X$ containing a general point of ${\rm Tan}(X)$. Let $w_X$ be their number.  It is a question, on which we will not dwell here, whether $w_X=1$ if $X$ is smooth and not secant defective. However, one may have $w_X>1$ when $X$ is either secant defective or singular: e.g., consider the cases $X$ is the Veronese surface of degree 4 in $\Ps^5$, where $w_X=2$, and $X$ is a surface lying on a $3$--dimensional, nondegenerate cone with vertex a line in $\Ps^5$, where $w_X$ can be as large as we want. 

The following lemma is  known to the experts. We give a proof for completeness.

\begin{lemma}\label{tangdeg} If $X\subset \Ps^r$, with $r\geq 2n$, is a smooth, irreducible, nondegenerate variety of dimension $n$, then $v_X=0$ if and only if $X$ is tangentially defective, whereas $v_X=t_Xw_X\geq t_X$ if $X$ is not tangentially defective.\end{lemma}

\begin{proof} Consider a general projection $\phi: X\to \Ps^{2n-1}$ and let $Z$ be the ramification scheme of $\phi$ on $X$. By the generality assumption about $\phi$, $Z$ is reduced of finite length $\ell_Z$. One has $\ell_Z=0$ if and only if $X$ is tangentially defective, whereas
$\ell_Z=t_Xw_X$ otherwise.  Look at the exact sequence
\begin{equation}\label{eq:rohn}
0\to \bigoplus_{r-2n+1}\mathcal O_X\to \mathcal N_{X,\Ps^r}(-1)\to \mathcal N_\phi(-1)\to 0
\end{equation}
where $\mathcal N_\phi$ is the \emph{normal sheaf} to the morphism $\phi$ (see \cite{cil}, p. 358, sequence (2.2)). This is locally free of rank $n-1$ off $Z$, where there is torsion, with lenght equal to $\ell_Z$.  Taking Chern classes of the sheaves in \eqref{eq:rohn}, the assertion follows. \end{proof}

\begin{remark}\label{rem:secdef} Let $X\subset \Ps^r$, with $r\geq 2n+1$, be a smooth, irreducible, nondegenerate of dimension $n$. Then $\delta_X=0$ if and only if $X$ is secant defective. 
If $H_X$ is a general hyperplane section of $X$, then $X$ is tangentially defective if and only if $H_X$ is secant defective, i.e. $v_X=0$ if and only if $\delta_{H_X}=0$. This is a consequence of Terracini's Lemma (see \cite {Zak}). 
\end{remark} 

Going back to the question about the validity of \eqref {extension}, the arguments in the proof of
Theorem \ref{clebsch}, based on Surface Theory, do not apply for $n\geq 3$. However,
comparing $\delta_X$ with the analogous number for a general hyperplane section of $X$, we may prove the
following:

\begin{proposition}\label{formula}
Let $X\subset\Ps^r$ be a nondegenerate, irreducible, smooth,
projective variety, of dimension $n\geq 2$ and degree $d$, with
$r\geq 2n+1$. Let  $Y:=H_X\subset \Ps^{r-1}$ be a general hyperplane
section of $X$ and $C\subset \Ps^{r-n+1}$ a general curve
section of $X$. Let $g$ and $e$ be the genus
and the index of speciality  of $C$. Then:

\medskip
$(i)$ $2(\delta_Y-\delta_X)=v_X$;

\medskip
$(ii)$ $\left[\delta_X-{\binom {d-n}
{2}}\right]-\left[\delta_Y-{\binom {d-(n-1)}
{2}}\right]=d-n-{\frac{1}{2}}v_X$;

\medskip
$(iii)$ $\delta_X\leq\delta_Y$, with equality if and only if $\delta_Y=0$;

\medskip
$(iv)$ $\left[\delta_X-{\binom {d-n}
{2}}\right]-\left[\delta_Y-{\binom {d-(n-1)} {2}}\right]\leq d-n$, with equality as in $(iii)$;

\medskip
$(v)$ $\delta_X-{\binom {d-n} {2}}\leq (n-2)(d-n)
+{\binom{n-2}{2}}$ with equality only if $n=2$ and $X$ is a rational scroll;

\medskip
$(vi)$ $\delta_X-{\binom {d-n} {2}}< -g + (n-1)(d-n)
+{\binom{n-1}{2}}$.

\medskip \noindent In particular, if either $g> (n-1)(d-n)
+{\binom{n-1}{2}}$, or $e\geq 2(n-1)$, then $\delta_X<{\binom
{d-n}{2}}$.
\end{proposition}

\begin{proof} Let $\varphi:X\to\Ps^{2n}$ and $\psi:Y\to\Ps^{2(n-1)}$ be
general  projections. Denote by $c(\mathcal
T_X)^{-1}=1+s_1+\dots+s_n$ the inverse total Chern class of $X$, with
$s_i\in A^i(X)$ (we abuse notation and denote in the same way elements
of $A^n(X)$ and their degree: we did this already a few times above).
By the double point formula (\cite{Fulton}, p.
166), $2(\delta_Y-\delta_X)$ is equal to
\begin{equation}\label{diffdelta}
\left[c(\varphi^*\mathcal T_{\Ps^{2n}})c(\mathcal
T_X)^{-1}\right]_n-\left[c(\psi^*\mathcal
T_{\Ps^{2(n-1)}})c(\mathcal
T_Y)^{-1}\right]_{n-1}=\sum_{i=0}^{n}{\binom{2n}{i}}H_X^is_{n-i}.
\end{equation}
Since $c(\mathcal
N_{X,\Ps^r})=(1+H_X)^{r+1}c(\mathcal T_X)^{-1}$, we have
\begin{equation}\label{cn}
c_n(\mathcal
N_{X,\Ps^r}(-1))=\sum_{i=0}^{n}(-1)^{n-i}{\binom{r-n-i}{n-i}}H_X^{n-i}c_i(\mathcal
N_{X,\Ps^r})
\end{equation}
$$
=\sum_{i=0}^{n}(-1)^{n-i}{\binom{r-n-i}{n-i}}H_X^{n-i}\left[\sum_{j=0}^{i}{\binom{r+1}{i-j}}H_X^{i-j}s_{j}\right]
$$
$$
=\sum_{i=0}^{n}\left[\sum_{h=0}^{i}(-1)^{i-h}{\binom{r-2n+i-h}{i-h}}{\binom{r+1}{h}}
\right]H_X^is_{n-i}.
$$
For any $i\in\{0,\dots,n\}$ and any
$r\geq 2n+1$ one has
$$
{\binom{2n}{i}}=\sum_{h=0}^{i}(-1)^{i-h}{\binom{r-2n+i-h}{i-h}}{\binom{r+1}{h}}.
$$
Comparing (\ref{diffdelta}) with (\ref{cn}) we obtain $(i)$.

Property $(ii)$  follows from $(i)$. Properties $(iii)$ and $(iv)$
follow from  $(i)$ and $(ii)$ and Remark \ref {rem:secdef}.

Iterating  $n-2$ times $(iv)$, and denoting by $S$ 
 the general surface section  of $X$, to which we apply Theorem \ref{clebsch}, we obtain
\begin{equation}\label{eq:eineq}
\delta_X-{\binom{d-n}{2}}\leq
\left[\delta_S-{\binom{d-2}{2}}\right]+
(n-2)(d-n)+{\binom{n-2}{2}}
\leq (n-2)(d-n)+{\binom{n-2}{2}}.
\end{equation}
If $n=2$, equality between the extremes 
holds only if $X$ is a rational scroll by Theorem \ref{clebsch}.
If $n>2$ the equality cannot hold. Otherwise
$S$ is a rational scroll, therefore also $X$ is a rational scroll
(see \cite {EHarris2}), hence the leftmost term in \eqref{eq:eineq}
is zero (see Remark \ref {rem:scrolls}), whereas the rightmost term is not.

Property $(vi)$ follows by applying $(ii)$
to $S$ in the middle term of \eqref {eq:eineq}, by 
recalling that $g={\binom{d-1}{2}}-\delta_C$ and noting that
$v_S>0$ because $S$, being smooth, is not tangentially defective
(see Lemma \ref {tangdeg} and   (5.37) of \cite 
{GH}).

For the final assertion  notice that if $e\geq 2(n-1)$ then $g\geq
(n-1)d +1$, hence $-g + (n-1)(d-n) +{\binom{n-1}{2}}<0$.
\end{proof}

\begin{remark}\label{finalr} $(i)$ In the setting of Proposition \ref {formula}, 
assume $X$ is not tangentially defective. Then none of its general 
linear section is tangentially defective.  
In view of Proposition \ref{formula}, $(ii)$,
one may ask whether 
\begin{equation}\label{conguno}
v_X\geq 2(d-n)
\end{equation}
or, rather
\begin{equation}\label{congunobis}
t_X\geq 2(d-n).
\end{equation}
If so, applying \eqref {conguno} to  the successive general linear sections of $X$
one would deduce 
\begin{equation}\label{congdue} \delta_X\leq{\binom {d-n}
{2}}-g.
\end{equation}

Note that, if $X$ is tangentially defective, then 
$\delta_X=0$ and \eqref {congdue} 
holds in this case, since Castelnuovo's bound implies 
$g<{\binom {d-n} {2}}$. 

In conclusion one is lead to the following:

\medskip
{\bf{Question.}} {\it{Let $X\subset\Ps^r$ be a nondegenerate,
irreducible, smooth, projective variety of dimension $n$ and
degree $d$, with $r\geq 2n+1$. Is it true that \eqref {congdue} holds,
with equality if and only if $g=0$, hence
$X$ is a rational  scroll?}}
\medskip

When $n=1$ the inequalities \eqref{conguno}, \eqref{congunobis} and
\eqref{congdue} are obvious.  In fact in this case $w_X=1$ and therefore $t_X=v_X$.
Moreover
$c_1(\mathcal
N_{X,\Ps^r}(-1))=K_X+2H_X$,  hence $v_X=t_X=2g-2+2d\geq 2(d-1)$, and ${\binom
{d-1} {2}}-\delta_X= g$ is  the Pl\"ucker--Clebsch
formula.

In  case $n=2$, in view of Proposition \ref{formula}, $(ii)$, one has $v_X>0$ because $X$, being smooth, is never tangentially defective
and, by \eqref{eq:vS}, \eqref{conguno} reads $K_X^2-6\chi(\mathcal O_X)+4g-2\geq 0$. 

If $X$ is smooth and
nondegenerate in $\Ps^4$ then (\ref{congdue}) becomes $g\leq{\binom {d-2} {2}}$, which holds by
Castelnuovo's bound. On the contrary, if $X\subset \Ps^3$
then $g={\binom {d-1} {2}}>{\binom {d-2} {2}}$ for $d\geq 3$, 
so in this case (\ref{congdue}) is false.

$(ii)$ In  case $n=2$, if $p_a(X)\leq 0$, e.g. if  $\kappa(X)\leq 0$, then (\ref{congdue}) holds.
In fact, taking into
account the proof of Proposition \ref{nspanned}, in order to prove
(\ref{congdue}) one may assume that $\mathcal O_X(K_X+H_X)$ is
spanned. In this case, by (\ref{rifconj}), one has
$$
\delta_X-{\binom {d-2} {2}}\leq 6p_a(X)-g-d+4\leq -g-d+4<-g
$$
as soon as $d\geq 5$ (the case $d=4$ is obvious).
Further cases in which we are able to prove (\ref{congdue}) are the following (we omit the proof):
complete intersections; surfaces contained in a threefold of
minimal degree;  surfaces contained in a smooth hypersurface of
$\Ps^5$ of degree $t$ with $d\geq t^3$ (in order to
prove (\ref{congdue}), one may assume  $X\subset \Ps^5$, and  any smooth surface in $\Ps^5$ is contained in
some smooth hypersurface); surfaces for which $3\chi(\mathcal
O_X)\geq d^2$; arithmetically Cohen-Macaulay surfaces contained in
a threefold of degree $s$, with $d\gg s$. Actually in the first
three cases we find  $\delta_X-{\binom {d-2} {2}}\leq -2g$,
and in a smooth threefold of minimal degree in $\Ps^5$ there are
surfaces $X$ for which $\delta_X-{\binom {d-2} {2}}= -2g$ with
$g>0$, and surfaces $X$ with $\delta_X-{\binom {d-2} {2}}> -3g$.

$(iii)$ By Proposition \ref{formula}, for  varieties $X$ (if any)
for which \eqref {congdue} fails, one has $g\leq (n-1)d$ and $e\leq 2(n-1)-1$. On the
other hand, using a similar argument as in the proof of
(\ref{pg}), one may prove that
\begin{equation}\label{acm}
p_g(X)\leq \sum_{i=n-1}^{+\infty}{\binom {i-1}
{n-2}}h^1(C,\mathcal O_{C}(i)).
\end{equation}
Since $h^1(C,\mathcal O_{C}(i))\leq g$, then
\begin{equation}\label{appr}
p_g(X)\leq \sum_{i=n-1}^{2(n-1)-1}{\binom {i-1} {n-2}}g \leq
{\binom {2(n-1)} {n-1}}(n-1)d=O(d).
\end{equation}

$(iv)$ As a consequence we can prove that {\it{if
$X\subset \Ps^r$ is arithmetically Cohen--Macaulay and $d\gg r$
then \eqref {congdue} holds with strict inequality}}. In fact, when  $X$ is
arithmetically Cohen-Macaulay, equality holds in (\ref{acm}),
and for any $i\geq 0$ one has $h^1(C,\mathcal
O_{C}(i))=\sum_{j=i+1}^{+\infty}{(d-h_{\Gamma}(j))}$
($\Gamma$ is the 
general $0$-dimensional linear section of  $X$ and $h_Z$ is, as usual, the \emph{Hilbert function} of a projective scheme $Z$). We
deduce
$$
p_g(X)=\sum_{i=n}^{+\infty}{\binom {i-1} {n-1}}(d-h_{\Gamma}(i)).
$$
When $\delta_X \geq {\binom {d-n} {2}}$, (\ref{appr}) 
says that $p_g(X)\leq O(d)$, so we have
$$
\sum_{i=n}^{+\infty}{\binom {i-1} {n-1}}(d-h_{\Gamma}(i))\leq
O(d).
$$
One sees that this is impossible if $d\gg r$.

$(v)$ With the same notation as in Proposition \ref{formula},
one has: {\it{if $n\geq 3$, $\delta_Y - {\binom {d-(n-1)}
{2}}\leq -g$ and $p_g(X)>0$, then \eqref {congdue} holds with strict inequality}}. 
In fact
(\ref{acm}) yields $e\geq 2$, thus $g>d-n$. Hence from
Proposition \ref{formula}, $(iv)$, we get $\delta_X-{\binom {d-n}
{2}}\leq  \delta_Y - {\binom {d-(n-1)} {2}}+d-n\leq -g+d-n<0$.
\end{remark}

\section{A stronger inequality}\label{sec:strong}

In this section we improve Theorem \ref{clebsch} as follows:

\begin{theorem}\label{clebsch2} With the same notation as in Theorem  \ref{clebsch},
one has
\begin{equation}\label{eq:improve}
\delta_S - {\binom {d-2} {2}}\leq -\frac{g}{2},
\end{equation}
and equality holds if and only if $S$ is a rational scroll in
$\Ps^r$.
\end{theorem}

\begin{remark}\label{new}
$(i)$ By Theorem \ref{clebsch}, it
suffices to prove the assertion when $g>0$. By (\ref{doublepoint}) this is
equivalent to prove that: {\it{when $g>0$, then
\begin{equation}\label{ncs}
K^2_S> 6p_a(S)-\frac{9}{2}g+8.
\end{equation}}}

$(ii)$ By the proof of Proposition \ref{nspanned} and Remark
\ref{finalr}, $(ii)$, we may  assume that $\mathcal O_S(K_S+H_S)$ is
spanned and  $p_a(S)>0$. In particular 
$g>\frac{d+1}{2}\geq 3$, otherwise $e\leq 0$, hence $p_a(S)\leq 0$
by (\ref{pg}). Moreover by (\ref{positivity}) we see that if
$g>2(d-2)$ then \eqref {eq:improve} holds with strict inequality. So
we may also assume $g\leq 2(d-2)$.
\end{remark}

\begin{lemma}\label{cs} If $g>0$ and $3p_a(S)<\frac{5}{2}g+d-9$
then \eqref {eq:improve} holds with strict inequality.
\end{lemma}
\begin{proof}
Consider the adjunction map
$\phi:S\to \Ps^R$. By Riemann--Roch one has
\begin{equation}\label{R}
R=h^0(S,\mathcal O_S(K_S+H_S))-1=p_a(S)+g-1
\end{equation}
since $h^i(S,\mathcal O_S(K_S+H_S))=0$, $1\leq i\leq 2$, by Kodaira vanishing theorem.
Let $\Sigma$ be the image
of $S$ via $\phi$ and let $\sigma$ be its degree. Except for a few cases in which $p_a(S)\leq 0$
(\cite{SVDV}, pg. 593-594), one knows that $\Sigma$ is a smooth surface, birational to $S$ via
$\phi$. In particular we have $\sigma\geq R-1$, i.e.
\begin{equation}\label{sigma}
\sigma=(K_S+H_S)^2=K_S^2+4g-4-d\geq R-1.
\end{equation}
Set
$$
\sigma-1=m(R-2)+\epsilon,\quad{\rm with }\quad 0\leq \epsilon\leq R-3.
$$

By (\ref{sigma}) we see that $m\geq 1$. The case
$m=1$ is not possible. In fact, by Castelnuovo Theory
\cite{EHarris, Harris} (see Remark \ref {finalr}, $(iii)$ and $(iv)$ above), we know that
\begin{equation}\label{Cbound}
p_g(\Sigma)\leq
\sum_{i=1}^{+\infty}{(i-1)}(\sigma-h_{\Gamma}(i)),
\end{equation}
where $\Gamma$ is the 
general $0$-dimensional linear section of  $\Sigma$. Moreover
\begin{equation}\label{Cbounddue}
\sigma\geq h_{\Gamma}(i)\geq \min\,\{\sigma,
i(R-2)+1\}\quad {\text{for any $i\geq 1$}}.
\end{equation}
If $m=1$ then $\sigma\leq 2(R-2)$, and from
(\ref{Cbound}) and (\ref{Cbounddue}) we get $p_g(\Sigma)=0$, against our assumption $p_g(\Sigma)=p_g(S)\geq
p_a(S)>0$ (compare also with (\cite{BS}, Lemma 10.1.2)).

When $m=2$,  by (\ref{Cbound}) and (\ref{Cbounddue}), we get
$$
p_a(S)\leq p_g(S)=p_g(\Sigma)\leq \sigma-h_\Gamma(2)\leq
\sigma-(2(R-2)+1)
$$
which implies $K_S^2\geq 3p_a(S)-2g+d-1$ (use (\ref{R}) and
(\ref{sigma})). Since  $3p_a(S)<\frac{5}{2}g+d-9$, we deduce
(\ref{ncs}), hence the assertion holds.

Finally assume $m\geq 3$. Since $\sigma\geq m(R-2)+1$, by
(\ref{R}) and (\ref{sigma}) we  have $ K_S^2\geq
mp_a(S)+(m-4)g+d+5-3m\geq 3p_a(S)-2g+d-1$
because $p_a(S)>0$ and $g\geq 3$ (see  Remark \ref{new}, $(ii)$). 
We conclude as in the case $m=2$.
\end{proof}

\begin{lemma}\label{tans}
Let $X\subset\Ps^r$ ($r\geq 2n+1$) be a nondegenerate, irreducible 
projective variety of dimension $n$ which is neither tangentially defective, nor secant defective. Then 
\begin{equation}\label{eq:bound}
t_X\geq 2(r-2n+1).
\end{equation}
\end{lemma}
\begin{proof} The main remark is that ${\rm Tan}(X)$ is singular along $X$,
as a local computation shows.

First we examine the case $r=2n+1$. Let $\ell$ be a general secant
line of $X$. Then $\ell$ contains two distinct
points of $X$, which are singular points of ${\rm Tan}(X)$.
Moreover $\ell$ is not contained in ${\rm Tan}(X)$, otherwise $X$ would
be secant defective against the assumption. It
follows that $t_X\geq 4$.

Next assume $r> 2n+1$ and argue by induction on $r$. Fix a general
point $x\in X$, and denote by $X'\subset \Ps^{r-1}$ the image of
$X$ via the projection $\phi$ from $x$. By the
Trisecant Lemma (\cite{Ciro}, Proposition 2.6, pg. 158), $\phi$ induces a
birational map of $X$ to $X'$. It also induces a generically finite map, of a
certain degree $\nu$, from ${\rm Tan}(X)$ to its image $V$. Otherwise
${\rm Tan}(X)$ would be a cone of vertex $x$ for a general $x\in X$, and
this would imply that $X$ is degenerate, for the set of vertices of a
cone is a linear space. In particular we have
$\dim(V)=\dim({\rm Tan}(X))=2n$. Since the general tangent space to $X$
projects to the general tangent space to $X'$, one has $V={\rm Tan}(X')$,
thus $X'$ is not tangentially defective.
The same argument says that
$\phi$ induces a generically finite map of
${\rm Sec}(X)$ to  ${\rm Sec}(X')$, thus $X'$, as well as $X$ is not secant defective.
By induction, we have
$t_{X'}\geq 2(r-2n)$.  Since $x$ is a
point of ${\rm Tan}(X)$ of multiplicity $\mu\geq 2$ we deduce
$t_X=\mu + \nu t_{X'}\geq 2 + t_{X'}\geq
2+2(r-2n)=2(r-2n+1)$.
\end{proof}

It is a nice problem, on which we do not dwell here, to determine all varieties $X$,
which are neither tangentially defective, nor secant defective, for which equality
holds in \eqref {eq:bound}. If $X$ is smooth and one believes that \eqref {congdue},
or rather \eqref {congunobis}, holds (see the discussion in Remark \ref {finalr}, $(i)$), 
then $X$ should conjecturally be a rational normal scroll. There are however singular
varieties $X$ reaching the bound \eqref {eq:bound}, which are not rational normal
scrolls.

\begin{corollary}\label{stime}
Let $S\subset\Ps^r$ ($r\geq 5$) be a nondegenerate, irreducible,
smooth, projective surface of degree $d$ and sectional genus $g$.
Let $C$ be a general hyperplane section of $S$. Then
$$
\quad \delta_S - {\binom {d-2} {2}}\leq
-p_a(S)+\sum_{i=2}^{+\infty}h^1(C,\mathcal O_C(i)).
$$
\end{corollary}
\begin{proof}
We may assume $S$ is linearly normal, i.e.
$r=h^0(S,\mathcal O_{S}(1))-1$.
On the other hand, from Proposition \ref {formula}, $(ii)$, we see that
$$
\delta_S - {\binom {d-2} {2}}=-g+d-2-\frac{1}{2}v_S.
$$
Moreover by
Lemma \ref{tans} we have
$$
v_S\geq t_S\geq 2(r-3)=
2(h^0(S,\mathcal O_{S}(1))-4).
$$
Hence 
\begin{equation}\label{ngen}
\delta_S - {\binom {d-2} {2}}\leq-g+d-2-(h^0(S,\mathcal
O_{S}(1))-4).
\end{equation}
Since $\chi(\mathcal O_S(1))=\chi(\mathcal O_S)+\chi(\mathcal
O_C(1))$, we also have
$$
h^0(S,\mathcal O_{S}(1))-4=h^1(S,\mathcal O_{S}(1))-h^2(S,\mathcal
O_{S}(1))+(p_a(S)+d-2-g).
$$
From (\ref{ngen}) we deduce
$$
\delta_S - {\binom {d-2} {2}}\leq -p_a(S)+h^2(S,\mathcal
O_{S}(1))-h^1(S,\mathcal O_{S}(1)).
$$
With an argument similar to the one used to prove (\ref{pg}), one sees that
$
h^2(S,\mathcal O_S(1))=h^0(S,\omega_S(-1))\leq
\sum_{i=2}^{+\infty}h^1(C,\mathcal O_C(i)).
$
The assertion follows. 
\end{proof}

We are now in position to prove Theorem \ref{clebsch2}. \medskip

\begin{proof}[Proof of Theorem \ref{clebsch2}] By Remark \ref{new}, $(ii)$, we may assume $\frac{d+1}{2}<g\leq
2(d-2)$.

First we examine the range $\frac{d+1}{2}<g\leq d$. In this case
$e\leq 1$. Put $g_m^s:=|\mathcal O_C(1)|$, and set
$$
2g-2=(k-1)m+h,\quad {\rm with}\quad 0\leq h\leq m-1.
$$
By Comessatti's bound \cite{Comessatti} one knows that
\begin{equation}\label{comes}
s\leq \frac{2k(m-1)-2g}{k(k+1)}+1.
\end{equation}
In our case $s=h^0(C,\mathcal O_C(1))-1$, $m=d$ and $k=2$.
So $h^0(C,\mathcal O_C(1))-1\leq
\frac{2}{3}d-\frac{1}{3}g+\frac{1}{3}$. By Riemann-Roch Theorem we
deduce $h^1(C,\mathcal O_C(1))\leq
\frac{2}{3}g-\frac{1}{3}d+\frac{1}{3}$. Hence from (\ref{pg}) we
get 
$$3p_a(S)\leq 2g-d+1<\frac{5}{2}g+d-9$$
(the second inequality holds since we may assume $d\geq 5$). This proves
Theorem
\ref{clebsch2} in the range $\frac{d+1}{2}<g\leq d$ by
Lemma \ref{cs}.

Next, assume $d<g\leq \frac{3d+1}{2}$. Then $e\leq 2$. By applying 
(\ref{comes}) to
$|\mathcal O_C(i)|$, with $1\leq i\leq 2$, 
we  prove that 
$h^1(C,\mathcal O_C(1))\leq \frac{5}{6}g-\frac{1}{2}d+\frac{1}{2}$
and $h^1(C,\mathcal O_C(2))\leq
\frac{2}{3}g-\frac{2}{3}d+\frac{1}{3}$. 
From (\ref{pg}) we get 
$$3p_a(S)\leq
\frac{9}{2}g-\frac{7}{2}d+\frac{5}{2}<\frac{5}{2}g+d-9.$$
The second inequality holds when $d>8$ because $g\leq \frac{3d+1}{2}$ and when
$d\leq 8$ by Castelnuovo's bound. This proves Theorem
\ref{clebsch2} in the range $d<g\leq \frac{3d+1}{2}$ by Lemma
\ref{cs}.

Finally assume $\frac{3d+1}{2}<g\leq 2(d-2)$. Then $e\leq 3$.
Using again (\ref{comes}) one sees that
$h^1(C,\mathcal O_C(2))\leq \frac{2}{3}g-\frac{2}{3}d+\frac{1}{3}$
and  $h^1(C,\mathcal O_C(3))\leq \frac{2}{3}g-d+\frac{1}{3}$. 
By Remark \ref{new}, $(i)$, and Corollary
\ref{stime},  we may assume
$$
p_a(S)\leq \frac{1}{2}g+h^1(C,\mathcal O_C(2))+h^1(C,\mathcal
O_C(3)).
$$
So 
$$3p_a(S)\leq \frac{11}{2}g-5d+2<\frac{5}{2}g+d-9$$
(the second inequality holds because $g\leq 2(d-2)$). This concludes the
proof of Theorem \ref{clebsch2}.\end{proof}


\begin{thebibliography}{}



\bibitem{BS} Beltrametti, M.C.,  Sommese, A.J.: {\it The adjunction theory of complex projective
varieties}, de Gruyter Expositions in Mathematics, 1995.



\bibitem{Ciro} Chiantini, L. Ciliberto, C.: {\it Weakly defective varieties},
Trans. Amer. Math. Soc. ${\mathbf{354}}$, No. 1, 151-178 (2002).

\bibitem {cil} Ciliberto, C.: {\it On the Hilbert scheme of curves of maximal genus in a projective space}, Math. Z. \textbf{194}, 351--363 (1987). 



\bibitem{Comessatti} Comessatti, A.: {\it Limiti di variabilit\`a della dimensione e dell'ordine d'una $g_n^r$
sopra una curva di dato genere}, Atti R. Ist. Veneto Sci. Lett.
Arti $\bf{74}$ (1914/1915), 1685-1709.

\bibitem{EHarris} Eisenbud, D., Harris, J.: {\it Curves in projective space},
Presse de l'Universit\'e de Montreal, 1982.

\bibitem{EHarris2} Eisenbud, D.,  Harris, J.:  {\it On varieties of minimal degree (a centennial account)}, in   ``Algebraic geometry'', Bowdoin, 1985 (Brunswick, Maine, 1985),  3--13, Proc. Sympos. Pure Math., \textbf{46}, Part 1, Amer. Math. Soc., Providence, RI, 1987.






\bibitem{Fulton} Fulton, W.: {\it Intersection theory}, Ergebnisse
der Mathematik und ihrer Grenzgebiete; 3.Folge, Bd. 2,
Springer-Verlag,  1984.


\bibitem{GH} Griffiths, P., Harris, J.: {\it Algebraic geometry and local differential geometry}, Ann. Scient. \'Ec. Norm. Sup. $4^e$ s\'erie, $\bf 12$, $n^o 3$ (1979), 355-452.

\bibitem{Harris} Harris, J.: {\it A bound on the geometric genus of projective varities},
Ann. Scuola Norm. Sup. Pisa Cl. Sci. (4), $\bf{8}$, 35-68 (1981).



\bibitem{SVDV} Sommese, A.J. - Van de Ven, A.: {\it On the Adjunction Mapping},
Math. Ann., $\bf{278}$, 593-603 (1987).

\bibitem {Zak} Zak, F. L.: {\it Tangents and secants of algebraic varieties}, Translated from the Russian manuscript by the author, Translations of Mathematical Monographs, \textbf{127}, American Mathematical Society, Providence, RI, 1993. viii+164. 





\end{thebibliography}
\end{document}